\documentclass[12pt,a4paper]{amsart}
\usepackage{amsmath,amssymb,color}
\usepackage{graphicx}

\numberwithin{equation}{section}

\makeatletter
\def\@settitle{\begin{center}\baselineskip14\p@\relax\bfseries{\large\@title}\thispagestyle{empty}\end{center}}
\def\@setauthors{%
  \begingroup
  \def\thanks{\protect\thanks@warning}%
  \trivlist
  \centering\footnotesize \@topsep30\p@\relax
  \advance\@topsep by -\baselineskip
  \item\relax
  \author@andify\authors
  \def\\{\protect\linebreak}%
  {\authors}%
  \ifx\@empty\contribs
  \else
    ,\penalty-3 \space \@setcontribs
    \@closetoccontribs
  \fi
  \endtrivlist
  \endgroup
}
\def\maketitle{\par
  \@topnum\z@ 
  \@setcopyright
  \thispagestyle{firstpage}
  \ifx\@empty\shortauthors \let\shortauthors\shorttitle
  \else \andify\shortauthors
  \fi
  \@maketitle@hook
  \begingroup
  \@maketitle
  \toks@\@xp{\shortauthors}\@temptokena\@xp{\shorttitle}%
  \toks4{\def\\{ \ignorespaces}}
  \edef\@tempa{%
    \@nx\markboth{\the\toks4
      \@nx
      {\the\toks@}}{\the\@temptokena}}%
  \@tempa
  \endgroup
  \c@footnote\z@
  \@cleartopmattertags
}
\makeatother

\newtheorem{Proposition}{Proposition}[section]
\newtheorem{Theorem}[Proposition]{Theorem}

\newtheorem{Lemma}[Proposition]{Lemma}

\newcommand{\ds}{\displaystyle}
\newcommand{\N}{\mathbb N}
\newcommand{\norm}[1]{ \lVert #1 \rVert }
\newcommand{\ulphi}{\underline{\varphi}}
\newcommand{\ul}[1]{\underline{#1}}
\newcommand{\ol}[1]{\overline{#1}}
\newcommand{\R}{\mathbb R}
\newcommand{\innV}[2]{\left\langle \mskip2.0 mu #1 \mskip2.0 mu | 	\mskip2.0 mu #2 \mskip2.0 mu \right\rangle}
\newcommand{\tm}{\triangle}
\newcommand{\Clanprod}[2]{#1 \,\triangle\, #2}
\newcommand{\ce}{c_{\ue}}
\newcommand{\ue}{\boldsymbol{\varepsilon}}
\newcommand{\pmat}[1]{\begin{pmatrix} #1 \end{pmatrix}}
\newcommand{\set}[2]{\left\{#1 ;\; #2 \right\}}
\newcommand{\cep}{c_{\ue'}}
\newcommand{\dtm}{\raisebox{1pt}{\rotatebox[origin=c]{180}{$\triangle$}}}
\newcommand{\dClanprod}[2]{#1\,\dtm\, #2}
\newcommand{\Sym}{\mathrm{Sym}}
\newcommand{\bs}[1]{\boldsymbol{#1}}

\newcommand{\bl}[1]{\textcolor{blue}{#1}}

\begin{document}

\title[Algebraic proof of explicit formulas of basic relative invariants]{
Algebraic proof of explicit formulas of basic relative invariants of homogeneous cones
}
\author[H. Nakashima]{Hideto NAKASHIMA}
\address{Graduate School of Mathematics, Nagoya University, Furo-cho, Chikusa-ku, Nagoya, Japan, 464--8602}
\email{h-nakashima@math.nagoya-u.ac.jp}
\theoremstyle{plain}
\begin{abstract}
The aim of this paper is to give another proof 
to a result on the image of a homogeneous quadratic map which is positive with respect to a homogeneous cone,
given by Graczyk and Ishi in 2014.
The new proof depends on a purely algebraic method,
whereas the original depends on analytic arguments.
This enables us to give
explicit formulas of the basic relative invariants of homogeneous cones, obtained by the previous paper [J.\ Lie Theory \textbf{24} (2014), 1013--1032]
without analytic arguments.
\end{abstract}
\maketitle
\section*{Introduction}

This paper is a continuation of our studies of
homogeneous open convex cones containing no entire line
(homogeneous cones for short in what follows)
in \cite{NN2013, N2014, N2018, N2020}.
In the study of homogeneous cones,
the basic relative invariants are fundamental objects
from analytic and algebraic points of view
(see \cite{Ishi1999,FK1994,N2020}, for example),
and explicit formulas of them are given in the previous paper~\cite{N2014}.
The proof there depends on a result obtained from 
Graczyk--Ishi~\cite{GI2014} and Ishi~\cite{Ishi2000},
which requires analytic arguments essentially for the proof.
It is however desirable to prove it without analytic arguments
because homogeneous cones provide many examples of non-reductive prehomogeneous vector spaces which can be defined over not only the real field but also general algebraic number fields 
(see~\cite{Kimura} for prehomogeneous vector spaces).
In this paper,
we give an algebraic proof to a part of the results in~\cite{GI2014}
which is sufficient to obtain explicit formulas of the basic relative invariants.

Let us describe the contents of this paper in more detail.
Let $V$ be a finite dimensional real vector space
with suitable inner product $\innV{\cdot}{\cdot}_V$
and $\Omega\subset V$ an open convex cone containing no entire line.
In what follows,
we always assume that $\Omega$ is homogeneous, that is,
the group $GL(\Omega):=\set{g\in GL(V)}{g\Omega=\Omega}$ acts on $\Omega$ transitively.
Vinberg~\cite{Vinberg} tells us that
there exists a split solvable Lie subgroup $H$ of $GL(\Omega)$
acting on $\Omega$ simply transitively.
In this paper,
we fix such an $H$,
and the action of $H$ on $V$ will be denoted by $\rho$.
We decompose $V$ as
$V=\bigoplus_{1\le j\le k\le r}V_{kj}$,
which is the normal decomposition of $V$
with respect to a complete orthogonal system $c_1,\dots,c_r$ of primitive idempotents
(see \eqref{eq:normal_decomp}).
Let $E$ be another finite dimensional real vector space
with inner product $\innV{\cdot}{\cdot}_E$,
and let $Q$ be a quadratic map on $E$ into $V$.
We say that $Q$ is $\Omega$-positive
if $Q[\xi]\in\overline{\Omega}\setminus\{0\}$ $(\xi\in E)$ whenever $\xi\ne 0$,
and that $Q$ is homogeneous
if for any $g\in GL(\Omega)$ there exists $A_g\in GL(E)$ such that
$g\bigl(Q[\xi]\bigr)=Q[A_g\xi]$ for all $\xi \in E$.
We extend $Q$ to the bilinear map by polarization in \eqref{def:polarization}, and use the same symbol $Q$.
By using inner products of $V$ and $E$,
we introduce a linear map $\varphi\colon V\to \mathrm{Sym}(E)$ by
\[
\innV{\varphi(x)\xi}{\eta}_E=\innV{Q(\xi,\eta)}{x}_V\quad
(x\in V,\ \xi,\eta\in E).
\]
The following theorem is what we need for the proof of explicit formulas of the basic relative invariants.

\renewcommand{\theProposition}{\Alph{Proposition}}
\begin{Theorem}[{Graczyk--Ishi~\cite{GI2014}, Ishi~\cite{Ishi2000}}]
\label{theo:GI2014}
For any $\Omega$-positive homogeneous quadratic map $Q$ on $E$,
there exists a unique $\ue\in\{0,1\}^r$ such that
the image $Q[E]$ of $Q$ is described as a closure of an $H$-orbit through $\ce=\varepsilon_1c_1+\cdots+\varepsilon_rc_r$, that is, one has
\[
Q[E]=\overline{\rho(H)\ce}.
\]
The vector $\bs{\varepsilon}=(\varepsilon_1,\dots,\varepsilon_r)$
is determined from $\dim\varphi(c_i)E$ $(i=1,\dots,r)$ and $\dim V_{kj}$
$(1\le j<k\le r)$.
\end{Theorem}
\renewcommand{\theProposition}{\arabic{section}.\arabic{Proposition}}

We note that,
in that paper~\cite{GI2014},
this theorem is proved by using a theory of Laplace transforms, 
and valid for more general quadratic maps called virtual quadratic maps.
On the other hand,
for the proof of explicit formulas of the basic relative invariants,
we need this theorem only for usual quadratic maps
so that we are able to give a proof to this theorem without analytic arguments.
We shall do so in Section~\ref{sect:proof} after preparing some terminologies and notations in Section~\ref{sect:preparation}.

\noindent
\textbf{Acknowledgments.}\rule{0pt}{2em}
The author is grateful to professor Takaaki Nomura
for the encouragement in writing this paper.

\section{Preliminaries}
\label{sect:preparation}

We begin this section with the definition of Vinberg algebras following Vinberg~\cite{Vinberg}.
Let $V$ be a finite dimensional real vector space with a bilinear product $\tm$.
We do not assume the commutativity and the associativity of the product, and also the existence of unit element.
For each $x\in V$,
let $L_x$ denote the left multiplication operator $L_xy=\Clanprod{x}{y}$ $(y\in V)$.
The pair $(V,\tm)$ is called a Vinberg algebra if the following three conditions are satisfied.
\begin{enumerate}
\item[(V1)] 
for any $x,y\in V$, one has $[L_x,L_y]=L_{\Clanprod{x}{y}-\Clanprod{y}{x}}$ (left symmetry)
\item[(V2)] 
there exists an $s\in V^*$ such that
$s(\Clanprod{x}{y})$ defines an inner product in $V$
(compactness)
\item[(V3)] 
for any $x\in V$, eigenvalues of $L_x$ are all real
(normality)
\end{enumerate}
Linear forms $s$ with the property (V2) are said to be \textit{admissible}.
In~\cite{Vinberg}, it is shown that
homogeneous convex domains correspond in a one-to-one way, 
up to isomorphisms,
to Vinberg algebras.
Moreover, homogeneous convex domains are cones if and only if
the corresponding algebras have a unit element.

We summarize properties of Vinberg algebras which we need later.
Let $V$ be a Vinberg algebra with unit element $e_V$.
We equip $V$ with an inner product $\innV{\cdot}{\cdot}_V$
defined by an admissible linear form $s_0\in V^*$,
that is,
\[
\innV{x}{y}_V:=s_0(\Clanprod{x}{y})\quad(x,y\in V).
\]
There exists a complete orthogonal system of primitive idempotents
$c_1,\dots,c_r$ with $c_1+\cdots+c_r=e_V$ such that
$V$ can be decomposed as
\begin{equation}
\label{eq:normal_decomp}
V=\bigoplus_{1\le j\le k\le r}V_{kj}
\end{equation}
where $V_{jj}=\R c_j$ $(j=1,\dots,r)$ and
\[
V_{kj}=
\set{x\in V}{L_{c_i}x=\tfrac{1}{2}(\delta_{ij}+\delta_{ik})x\text{ and } 
\Clanprod{x}{c_i}=\delta_{ij}x\ (i=1,\dots,r)}
\]
for $1\le j<k\le r$.
According to this decomposition,
we have the following multiplication rules:
\begin{equation}
\label{table}
\begin{array}{c}
\Clanprod{V_{ji}}{V_{lk}}=\{0\}\ (\text{if }i\ne k,l),\quad
\Clanprod{V_{kj}}{V_{ji}}\subset V_{ki}\ (i\le j\le k),\\[0.3em]
\Clanprod{V_{ji}}{V_{ki}}\subset V_{jk}\text{ or }V_{kj}\ 
(\text{according to $j\le k$ or $k\le j$, where $i\le j,k$}).
\end{array}
\end{equation}
Weight spaces $V_{kj}$ $(1\le j\le k\le r)$ are mutually orthogonal with respect to the inner product $\innV{\cdot}{\cdot}_V$.
Note that we have $s_0(c_j)>0$ for each $j=1,\dots,r$ because $s_0(c_j)=s_0(\Clanprod{c_j}{c_j})=\innV{c_j}{c_j}>0$.
 
By (V1) and (V3),
the space $\mathfrak{h}:=\set{L_x}{x\in V}$
of left multiplication operators forms a split solvable Lie algebra,
which is linearly isomorphic to $V$.
Let $H:=\exp\mathfrak{h}$ be the connected and simply connected Lie group corresponding to $\mathfrak{h}$.
Then, by \cite{Vinberg},
the $H$-orbit $\Omega$ through $e_V$ is a proper open convex cone in $V$,
and $H$ acts on $\Omega$ linearly and simply transitively.
By introducing lexicographic order among the subspaces $V_{kj}$,
we see that every $L_x$ $(x\in V)$ is simultaneously expressed
by a lower triangular matrix by~\eqref{table}.
Using this,
we introduce a coordinate system on $H$
according to Ishi~\cite[Section 2]{Ishi2000} as follows.
For each $h\in H$,
there exist unique $t_j\in \R$ $(j=1,\dots,r)$ and $L_j\in \bigoplus_{k>j}V_{kj}$ $(j=1,\dots,r-1)$ such that
\begin{equation}
\label{def:coordinate H}
h=\exp(t_1L_{c_1})(\exp L_1)\exp(t_2L_{c_1})\cdots\exp(L_{r-1})\exp(t_rL_{c_r}).
\end{equation}

Let $\Omega^*$ be the dual cone of $\Omega$, that is,
\[
\Omega^*=\set{y\in V}{\innV{x}{y}_V>0\text{ for all }x\in\overline{\Omega}\setminus\{0\}}.
\]
Since $H$ acts on $\Omega^*$ simply transitively through
the contragradient representation $\rho^*$ of $\rho$,
we see that $\Omega^*$ is also a homogeneous cone in $V$.
We denote by $(V,\dtm)$ the Vinberg algebra corresponding to $\Omega^*$.
Then, the following relationship holds between the products $\tm$ and $\dtm$:
\[
\innV{\dClanprod{x}{y}}{z}_V=\innV{y}{\Clanprod{x}{z}}_V\quad
(x,y,z\in V).
\]
Moreover, \cite[Proposition 1.2]{N2014} tells us that
\begin{equation}
\label{eq:relation of tm and dtm}
\Clanprod{x}{y}-\Clanprod{y}{x}
=
\dClanprod{y}{x}-\dClanprod{x}{y}\quad
(x,y\in V).
\end{equation}

Let $\Sym(m,\R)$ be the space of symmetric matrices of size $m$.
For $x=(x_{ij})\in\Sym(m,\R)$, we set
\[
\ul{x}:=\begin{cases}
\frac{1}{2}x_{ii}&(i=j),\\
x_{ij}&(i>j),\\
0&(i<j),
\end{cases}
\quad
\ol{x}:={}^{t\!}(\ul{x}).
\]
Namely, $\ul{x}$ is a lower triangular matrix such that
$\ul{x}+{}^{t\!}(\ul{x})=x$.
The product of Vinberg algebra $(\Sym(m,\R),\tm)$ is given as
\[
\Clanprod{x}{y}=\ul{x}y+y\ol{x}\quad(x,y\in\Sym(m,\R)),
\]
whereas that of $(\Sym(m,\R),\dtm)$ is
\[
\dClanprod{x}{y}=\ol{x}y+y\ul{x}\quad(x,y\in \Sym(m,\R)).
\]

Let $\varphi$ be a linear map from $V$ to $\Sym(m,\R)$.
The map $\varphi$ is called a representation of $(V,\dtm)$
if $\varphi$ is an algebra homomorphism 
\[
\varphi(\dClanprod{x}{y})=\ol{\varphi}(x)\varphi(y)+\varphi(y)\ul{\varphi}(x)\quad(x,y\in V)
\]
with a condition $\varphi(e_V)=I_m$.
Here, we set $\ul{\varphi}(x)=\ul{\varphi(x)}$ and $\ol{\varphi}(x)=\ol{\varphi(x)}$.

Let $E$ be a finite dimensional vector space with an inner product $\innV{\cdot}{\cdot}_E$
and $Q$ an $\Omega$-positive homogeneous quadratic map on $E$.
We extend $Q$ (and use the same symbol) to a bilinear map $Q\colon E\times E\to V$ by polarization, that is,
\[
Q(\xi_1,\xi_2)=\frac{1}{2}\bigl(
Q[\xi_1+\xi_2]-Q[\xi_1]-Q[\xi_2]
\bigr)
\quad(\xi_1,\xi_2\in E).
\]
Let $\varphi$ be a linear map from $V$ to $\mathrm{Sym}(E)$ defined by
\begin{equation}
\label{def:polarization}
\innV{\varphi(x)\xi}{\eta}_E=\innV{x}{Q(\xi,\eta)}_V
\quad(x\in V,\ \xi,\eta\in E).
\end{equation}
Then, $\varphi$ is a representation of $\Omega^*$ in the sense of Rothaus~\cite{Rothaus}, and
by~\cite[Theorem 2]{Ishi2011},
we can assume that $\varphi$ is a representation of the Vinberg algebra $(V,\dtm)$
without loss of generality.
Using this, 
we equip $W:=E\oplus V$ with a bilinear product $\tm$ by
\begin{equation}
\label{eq:prod W}
\Clanprod{(\xi_1+x_1)}{(\xi_2+x_2)}:=\ul{\varphi}(x_1)\xi_2+(2\,Q(\xi_1,\xi_2)+\Clanprod{x_1}{x_2})
\end{equation}
for $\xi_1,\xi_2\in E$ and $x_1,x_2\in V$.
Here, $\tm$ in the right hand side is the product of $(V,\tm)$.
Then, $(W,\tm)$ is a Vinberg algebra with no unit element
(cf.\ \cite[\S3]{NN2013}).
Notice that we have $\Clanprod{\xi_1}{\xi_2}=\Clanprod{\xi_2}{\xi_1}$ for any $\xi_1,\xi_2\in E$.
The inner product of $W$ is defined by using the admissible linear form $s_0$ of $V$ as
\begin{equation}
\label{eq:innW}
\begin{array}{r@{\ }c@{\ }l}
\innV{\xi_1+x_1}{\xi_2+x_2}_W
&:=&
s_0(2\,Q(\xi_1,\xi_2)+\Clanprod{x_1}{x_2})\\
&=&
\innV{x_1}{x_2}_V+2\innV{\xi_1}{\xi_2}_E.
\end{array}
\end{equation}
Note that $\norm{\xi}_W^2=2\norm{\xi}_E^2$.

\section{Algebraic proof}
\label{sect:proof}

In this section,
we give an algebraic proof to Theorem~\ref{theo:GI2014}.
Let $V$ be a Vinberg algebra corresponding to a homogeneous cone $\Omega$ of rank $r$,
and we keep all notations from the previous sections.
First, we decompose $V$ into three subspaces as follows:
\[
V=\R c_1\oplus V^{[1]}\oplus V',\quad
V^{[1]}:=\bigoplus_{k=2}^rV_{k1},\quad
V':=\bigoplus_{2\le j\le k\le r}V_{kj}.
\]
If we restrict the product $\tm$ of $V$ to $V'$ (and use the same notation),
then $(V',\tm)$ is also a Vinberg algebra so that
there exist a homogeneous cone $\Omega'$ corresponding to $V'$
and a split solvable Lie group $H'$ acting on $\Omega'$ simply transitively.
As in \cite[\S2]{N2018},
each element $h'$ in $H'$ can be written by removing the terms corresponding to $\R c_1$ and $V^{[1]}$ from~\eqref{def:coordinate H}, that is, there exist $h_j\in\R_{>0}$ $(j=2,\dots,r)$ and $v_{kj}\in V_{kj}$ $(2\le j<k\le r)$ such that
\[
h'=(\exp T_2)(\exp L_2)(\exp T_3)\cdots(\exp L_{r-1})(\exp T_r),
\]
where $T_j=(2\log h_j)L_{c_j}$ and $L_j=\sum_{k>j}L_{v_{kj}}$.
Moreover, if we set $Q'[\xi]:=\Clanprod{\xi}{\xi}$ for $\xi\in V^{[1]}$,
then $Q'$ is an $\Omega'$-positive homogeneous quadratic map on $V^{[1]}$.

Let $Q$ be an $\Omega$-positive homogeneous quadratic map on $E$,
and $\varphi$ the corresponding linear map 
$\varphi\colon V\to\mathrm{Sym}(E)$.
Since $c_1,\dots,c_r$ satisfy $\Clanprod{c_j}{c_k}=\delta_{jk}$,
we see that $\varphi(c_1),\dots,\varphi(c_r)$ are orthogonal projection on $E$.
By setting
$E_i:=\varphi(c_i)E$ for $i=1,\dots,r$,
we decompose $E$ into
\[
E=E_1\oplus E',\quad
E'=\bigoplus_{k=2}^rE_k.
\]
For each $\nu\in E$, we set $\xi_\nu:=\varphi(c_1)\nu$,
the orthogonal projection of $\nu$ to $E_1$.
We note that
the Vinberg algebra $W=E\oplus V$ can be described in a formal matrix form as
\[
W=\pmat{0&E_1&E'\\ E_1&\R c_1&V^{[1]}\\ E'&V^{[1]}&V'}.
\]
For later use,
we summarize multiplication rules between the spaces $E_1$, $E'$ and $V^{[1]}$
as a lemma, which can be obtained easily from \eqref{table}.

\begin{Lemma}
\label{lemma:product}
With respect to $E_1$, $E'$ and $V^{[1]}$,
one has the following multiplication table.
\[
\raisebox{-12pt}{\rm left}\,
\begin{array}{r|rrr}
\multicolumn{4}{c}{\quad{\rm right}}\\
&E_1&E'&V^{[1]}\\ \hline
\rule{0pt}{13pt}
E_1&\R c_1&V^{[1]}&0\\
\rule{0pt}{11pt}
E'&V^{[1]}&V'&0\\
\rule{0pt}{11pt}
V^{[1]}&E'&0&V'
\end{array}
\]
Here,
left factor of the products are placed in row entries,
and right ones in column entries.
\end{Lemma}

Suppose $E_1\ne\{0\}$.
We take and fix a non-zero $\xi\in E_1$.
Using this $\xi$,
we introduce two linear maps $r_\xi\colon V^{[1]}\to E'$ and 
$r^*_\xi\colon E'\to V^{[1]}$ by
\[
r_\xi(v)=\Clanprod{v}{\xi}\quad(v\in V^{[1]}),\quad
r^*_\xi(a)=\Clanprod{a}{\xi}\quad(a\in E').
\]

\begin{Lemma}
\label{lemma:orthogonal prod}
For any $a\in \mathrm{Image}\,r_\xi$ and $b\in \mathrm{Ker}\,r^*_\xi$, one has $\Clanprod{a}{b}=0$.
\end{Lemma}
\begin{proof}
Since $a\in \mathrm{Image}\,r_\xi$, 
there exists $v\in V^{[1]}$ such that $a=r_\xi(v)=\Clanprod{v}{\xi}$.
The condition (V1) yields that
\[
\Clanprod{a}{b}
=
\Clanprod{(\Clanprod{v}{\xi})}{b}
=
\Clanprod{(\Clanprod{\xi}{v})}{b}
+
\Clanprod{v}{(\Clanprod{\xi}{b})}
-
\Clanprod{\xi}{(\Clanprod{v}{b})}.
\]
By Lemma~\ref{lemma:product},
we know that $\Clanprod{\xi}{v}=0$ and $\Clanprod{v}{b}=0$ so that
\[
\Clanprod{a}{b}
=
\Clanprod{v}{(\Clanprod{\xi}{b})}.
\]
Since $\xi,b\in E$, the equation~\eqref{eq:prod W} tells us that $\Clanprod{\xi}{b}=\Clanprod{b}{\xi}=r^*_\xi(b)$,
and hence
we conclude $\Clanprod{a}{b}=0$ by $b\in\mathrm{Ker}\,r^*_\xi$.
\end{proof}

\begin{Lemma}
\label{lemma:direct sum}
One has $E'=\mathrm{Image}\,r_\xi\oplus \mathrm{Ker}\,r^*_\xi$
$($direct sum$)$.
\end{Lemma}

\begin{proof}
At first, we consider the composition $r^*_\xi\circ r_\xi$.
The left symmetry (V1) of the product yields that,
for $v\in V^{[1]}$
\[
r^*_\xi\circ r_\xi(v)
=
\Clanprod{(\Clanprod{v}{\xi})}{\xi}
=
\Clanprod{(\Clanprod{\xi}{v})}{\xi}
+
\Clanprod{v}{(\Clanprod{\xi}{\xi})}
-
\Clanprod{\xi}{(\Clanprod{v}{\xi})}.
\]
Lemma~\ref{lemma:product} tells us that
$\Clanprod{v}{\xi}\in E'\subset E$ so that
\eqref{eq:prod W} yields that
\[
\Clanprod{\xi}{(\Clanprod{v}{\xi})}=\Clanprod{(\Clanprod{v}{\xi})}{\xi}.\]
Thus, 
since $\Clanprod{\xi}{v}=0$ again by Lemma~\ref{lemma:product}, 
we have by the left symmetry
\[
\Clanprod{(\Clanprod{v}{\xi})}{\xi}
=
\Clanprod{v}{(\Clanprod{\xi}{\xi})}
-
\Clanprod{(\Clanprod{v}{\xi})}{\xi},
\]
whence $\Clanprod{(\Clanprod{v}{\xi})}{\xi}=
\frac{1}{2}\Clanprod{v}{(\Clanprod{\xi}{\xi})}$.
We note that
$\Clanprod{\xi}{\xi}=\frac{2\norm{\xi}_E^2}{s_0(c_1)}c_1$.
In fact,
Lemma~\ref{lemma:product} tells us that
there exists a suitable number $A$ such that $\Clanprod{\xi}{\xi}=Ac_1$, and then by \eqref{eq:innW}
\[
As_0(c_1)=s_0(Ac_1)=s_0\bigl(\Clanprod{\xi}{\xi}\bigr)=\norm{\xi}_W^2=2\norm{\xi}_E^2.
\]
Since $\Clanprod{v}{c_1}=v$ by definition of $V^{[1]}$,
we obtain
\[
r^*_\xi\circ r_\xi(v)
=
\Clanprod{(\Clanprod{v}{\xi})}{\xi}
=
\frac{1}{2}\,\Clanprod{v}{(\Clanprod{\xi}{\xi})}
=
\frac{\norm{\xi}_E^2}{s_0(c_1)}\,v,
\]
that is,
$r^*_\xi\circ r_\xi$ is a scalar map.
This means that
$r_\xi$ is injective while $r^*_\xi$ is surjective.
In particular, any $x\in E'$ can be written as
\[x=a+b\quad(a\in \mathrm{Image}\,r_\xi,\ b\in \mathrm{Ker}\,r^*_\xi).\]
In fact,
let $a=\frac{s_0(c_1)}{\norm{\xi}_E^2}r_\xi(r^*_\xi(x))\in\mathrm{Image}\,r_\xi$.
Then, since $r^*_\xi$ is linear and $r^*_\xi\circ r_\xi$ is a scalar map,
we have
\[
r^*_{\xi}(x-a)=r^*_\xi(x)-\frac{s_0(c_1)}{\norm{\xi}_E^2}r^*_\xi\circ r_\xi\circ r^*_\xi(x)
=
r^*_\xi(x)- r^*_\xi(x)=0,
\]
which implies $b:=x-a\in\mathrm{Ker}\,r^*_\xi$.

Next, 
let us take $x\in\mathrm{Image}\,r_\xi\cap \mathrm{Ker}\,r^*_\xi$.
Then,
Lemma~\ref{lemma:orthogonal prod}
together with \eqref{eq:innW} tells us that
\[
\norm{x}_W^2=\innV{x}{x}_W=s_0(\Clanprod{x}{x})=0,
\]
whence $\mathrm{Image}\,r_\xi\cap \mathrm{Ker}\,r^*_\xi=\{0\}$.
Now there is nothing to prove.
\end{proof}

This lemma implies that
$\mathrm{Image}\,r_\xi$ can be regarded as a copy of $V^{[1]}$ in $E'$ so that
we rewrite $\mathrm{Image}\,r_\xi$ and $\mathrm{Ker}\,r^*_\xi$ as
\[
E_{V^{[1]}}=E_{V^{[1]}}(\xi):=\mathrm{Image}\,r_\xi,\quad
E_{V^{[1]}}^\perp=E_{V^{[1]}}^\perp(\xi):=\mathrm{Ker}\,r^*_\xi.
\]
For each $x\in V'$,
let $\widetilde{\varphi}(x)$ denote the restriction of $\varphi(x)$ to $E_{V^{[1]}}^\perp$.
Then, $(\widetilde{\varphi},E_{V^{[1]}}^\perp)$ is a representation of $(V',\dtm)$.
In fact, 
Lemma~\ref{lemma:product} yields that
$\Clanprod{b}{x}=0$ and $\Clanprod{x}{\xi}=0$
for $b\in E_{V^{[1]}}^\perp$ and $x\in V'$ so that
we obtain by the left symmetry (V1)
\[
\begin{array}{r@{\ }c@{\ }l}
\Clanprod{(\Clanprod{x}{b})}{\xi}
&=&
\Clanprod{(\Clanprod{b}{x})}{\xi}
+
\Clanprod{x}{(\Clanprod{b}{\xi})}
-
\Clanprod{b}{(\Clanprod{x}{\xi})}\\
&=&
\Clanprod{x}{(r^*_\xi(b))}=0,
\end{array}
\]
whence $\widetilde{\varphi}(x)$ preserves $E_{V^{[1]}}^\perp$ for any $x\in V'$.
Let $\widetilde{Q}$ be the corresponding quadratic map.
Then, it is $\Omega'$-positive and homogeneous.

Let $\widetilde{Q}$ denote the restriction of $Q$ to $E_{V^{[1]}}^\perp$, that is,
\[
\widetilde{Q}[b]:=Q[b]\quad(b\in E_{V^{[1]}}^\perp).
\]

\begin{Lemma}
\label{lemma:tildeQ}
The quadratic map $\widetilde{Q}$ is $\Omega'$-positive and homogeneous.
\end{Lemma}

\begin{proof}
Since $\Clanprod{E'}{E'}\subset V'$, 
we see that $\widetilde{Q}$ is $\Omega'$-positive and thus we shall prove homogeneity.
Since $W':=E_{V^{[1]}}^\perp\oplus V'$ is a subspace of $W$
and $W$ satisfies the product $\eqref{eq:prod W}$, 
it is enough to show $\Clanprod{V'}{E_{V^{[1]}}^\perp}\subset E_{V^{[1]}}^\perp$.
By Lemma~\ref{lemma:product}, we have
$\Clanprod{b}{x}=0$ and $\Clanprod{x}{\xi}=0$
for $b\in E_{V^{[1]}}^\perp$ and $x\in V'$ so that
the left symmetry (V1) yields
\[
\begin{array}{r@{\ }c@{\ }l}
\Clanprod{(\Clanprod{x}{b})}{\xi}
&=&
\Clanprod{(\Clanprod{b}{x})}{\xi}
+
\Clanprod{x}{(\Clanprod{b}{\xi})}
-
\Clanprod{b}{(\Clanprod{x}{\xi})}\\
&=&
\Clanprod{x}{(r^*_\xi(b))}=0,
\end{array}
\]
which is what we want to prove.
\end{proof}

\begin{Proposition}
\label{prop:main}
For $\nu\in E$ with $\xi=\xi_\nu\ne 0$,
we decompose it according to the direct sum decomposition $E=E_1\oplus E_{V^{[1]}}\oplus E_{V^{[1]}}^\perp$ as 
\begin{equation}
\label{eq:decom xi}
\nu=\xi+a+b\quad(\xi=\xi_\nu\in E_1,\ a\in E_{V^{[1]}},\ b\in E_{V^{[1]}}^\perp).
\end{equation}
let $t>0$ and $u\in V^{[1]}$ be 
\[
e^t:=\frac{2\norm{\xi}_E^2}{s_0(c_1)},\quad
u:=2e^{-t/2}r^*_\xi(a)=2\sqrt{\frac{s_0(c_1)}{2\norm{\xi}^2_E}}\,r^*_\xi(a).
\]
Then, one has
\[
Q[\nu]=\rho(\exp L_{tc_1}\exp L_u)(c_1+\widetilde{Q}[b]).
\]
\end{Proposition}
\begin{proof}
Since $Q(\xi_1,\xi_2)$ is bilinear and symmetric, we have
\[
\begin{array}{r@{\ }c@{\ }l}
Q[\nu]
&=&
\ds
\frac{2\norm{\xi}^2_E}{s_0(c_1)}c_1+2\Clanprod{\xi}{(a+b)}+Q[a+b]\\
&=&
e^tc_1+e^{t/2}\cdot 2e^{-t/2}r^*_\xi(a)+Q[a]+\widetilde{Q}[b].
\end{array}
\]
On the other hand,
we have for general $t'\in\R$, $v\in V^{[1]}$ and $w\in V'$
\[
\rho(\exp L_{t'c_1} \exp L_{v'})(c_1+w)=
e^{t'}c_1+e^{t'/2} v'+(\tfrac{1}{2}\Clanprod{v'}{v'}+w).
\]
These observations tell us that it is enough to show $\frac{1}{2}\Clanprod{u}{u}=Q[a]$.
Since $a\in E_{V^{[1]}}$,
there exists a unique $v\in V^{[1]}$ such that $a=\Clanprod{v}{\xi}$.
Recalling that $r^*_\xi\circ r_\xi$ is a scalar map
as in the proof of Lemma~\ref{lemma:direct sum},
we have
\[
\frac{e^{t/2}}{2}\,u
=
\Clanprod{a}{\xi}
=
\Clanprod{(\Clanprod{v}{\xi})}{\xi}
=
r^*_\xi\circ r_\xi(v)
=
\frac{\norm{\xi}_E^2}{s_0(c_1)}v
=
\frac{e^t}{2}v.
\]
This implies that $u=e^{t/2}v$ and hence by definition of $v$ we obtain
\begin{equation}
\label{eq:prop 2.4}
\Clanprod{u}{\xi}=e^{t/2}\Clanprod{v}{\xi}=e^{t/2}a.
\end{equation}
On the other hand, the left symmetry (V1) yields that
\[
\begin{array}{r@{\ }c@{\ }l}
\displaystyle
\Clanprod{u}{u}
&=&
\displaystyle
2e^{-t/2}
\Clanprod{u}{(\Clanprod{a}{\xi})}\\
&=&
\displaystyle
2e^{-t/2}\bigl(
\Clanprod{a}{\bigl(\Clanprod{u}{\xi}\bigr)}
+
\Clanprod{\bigl(\Clanprod{u}{a}\bigr)}{\xi}
-
\Clanprod{\bigl(\Clanprod{a}{u}\bigr)}{\xi}
\bigr),
\end{array}
\]
and since $\Clanprod{u}{a}=\Clanprod{a}{u}=0$ by Lemma~\ref{lemma:product},
we obtain by \eqref{eq:prop 2.4}
\[
\frac{1}{2}\Clanprod{u}{u}
=
e^{-t/2}\,
\Clanprod{a}{\bigl(\Clanprod{u}{\xi}\bigr)}
=
\Clanprod{a}{a}
=
Q[a],\]
which proves the proposition.
\end{proof}

We now start giving another proof to Theorem~\ref{theo:GI2014}.
It is proceeded by the induction on rank $r$.
If $r=1$, then there is nothing to prove.
Therefore, let $r\ge 2$ and assume that the theorem holds for $r-1$.
Assume that $E_1=\{0\}$.
Then, $Q$ is an $\Omega'$-positive homogeneous quadratic map
so that by the induction hypotheses there exists $\ue'={}^t(\varepsilon_2,\dots,\varepsilon_r)$ such that
$Q[E]=\overline{\rho(H')\cep}$.
It can be embedded in $V$ by $Q[E]=\overline{\rho(H)\ce}$ with $\ue:={}^t(0,\ue')$,
and hence in this case the theorem holds.
In what follows, 
let us consider the case $E_1\ne\{0\}$.
Take and fix an arbitrary $\nu\in E$.
If $\xi=\xi_\nu=0$, 
then we take and fix a non-zero element $\eta\in E_1$.
For an integer $n\in \N$, we introduce a sequence $\{\nu_n\}$ in $E$ by 
\[
\nu_n=\xi_n+a+b,\quad
\xi_n=\begin{cases}
\xi&(\xi\ne 0)\\
\frac{1}{n}\eta&(\xi=0)
\end{cases}
\]
where $a\in E_{V^{[1]}}(\xi_n)$ and $b\in E_{V^{[1]}}^\perp(\xi_n)$.
Note that 
the spaces $E_{V^{[1]}}(\xi_n)$ and $b\in E_{V^{[1]}}^\perp(\xi_n)$ do not depend on the choice of $n$
and hence so are $a,b$,
but they depend on the choice of $\eta$.
Since $\xi_n\ne 0$, 
we can apply Proposition~\ref{prop:main} to $\nu_n$ so that
\[
Q[\nu_n]
=
\rho(\exp L_{t_nc_1}\exp L_{u_n})(c_1+\widetilde{Q}[b])
\]
where
$e^{t_n}=\frac{2\norm{\xi}^2_E}{s_0(c_1)}$ and
$u_n=2e^{-t_n/2}r^*_{\xi_n}(a)$.
Since $\widetilde{Q}$ is an $\Omega'$-positive homogeneous quadratic map on $E_{V^{[1]}}^\perp$ by Lemma~\ref{lemma:tildeQ},
the hypotheses of the induction shows that
the image of $\widetilde{Q}$ can be described as a closure of an $H'$-orbit, that is,
there exists a unique $\ue'={}^{t}(\varepsilon_2,\dots,\varepsilon_r)\in\{0,1\}^{r-1}$ such that
\[\widetilde{Q}[b]\in\overline{\rho(H')\cep}\qquad
(\cep:=\varepsilon_2 c_2+\cdots+\varepsilon_rc_r)\]
for all $b\in E_{V^{[1]}}^\perp$.
Note that $\ue'$ is determined by $\dim\varphi(c_j)E_{V^{[1]}}^\perp=\dim \varphi(c_j)E-\dim V_{j1}$ $(j=2,\dots,r)$ and by $\dim V_{kj}$ $(2\le j<k\le r)$.
This means that we can choose a sequence $\{h'_n\}_{n=1,2,\dots}$ in $H'$ to be $\displaystyle\lim_{n\to+\infty}\rho(h'_n)\cep=\widetilde{Q}[b]$ so that
we have
\[
c_1+\widetilde{Q}[b]
=
\lim_{n\to+\infty}
\bigl(
c_1+\rho(h'_n)\cep
\bigr)
=
\lim_{n\to+\infty}
\rho(h'_n)
\bigl(
c_1+\cep
\bigr).
\]
Here, we use a fact $\rho(h')c_1=c_1$ for $h'\in H'$ in the last equality.
Thus, if we put
\[
h_n:=(\exp L_{t_nc_1})(\exp L_{u_n})h'_n\in H\text{ and }
\ue={}^{t\!}(1,\varepsilon_2,\dots,\varepsilon_r)\in\{0,1\}^r,
\]
then we obtain by continuity of $Q$ and by an obvious fact $\displaystyle\lim_{n\to+\infty}\xi_n=\xi$
\[
\begin{array}{r@{\ }c@{\ }l}
Q[\nu]
&=&
\displaystyle
\lim_{n\to+\infty}Q[\nu_n]
=\lim_{n\to+\infty}\rho(\exp L_{t_nc_1}\exp L_{u_n})\rho(h'_n)(c_1+\cep)\\
&=&
\displaystyle
\lim_{n\to+\infty}\rho(h_n)\ce,
\end{array}
\]
whence $Q[\nu]\in \overline{\rho(H)\ce}$ for any $\nu\in E$.
Notice that the value $\varepsilon_1$ is determined according to $\dim E_1=0$ or not.

At last,
we need to show that $\ue\in\{0,1\}^r$ above does not depend on the choice of $\nu$.
Let us take two elements $\eta_1,\eta_2\in E_1$, and set
\[
E_k(\eta_i):=\mathrm{Image}\left(
r_{\eta_i}|_{V_{k1}}
\right)\quad(i=1,2).
\]
Then,
as similar to the proof of Lemma~\ref{lemma:direct sum},
the space $E_k(\eta_i)$ for each $\eta_i$ is linearly isomorphic to $V_{k1}$
through $r_{\eta_i}|_{V_{k1}}$ so that
we have $\dim E_k(\eta_1)=\dim E_k(\eta_2)$.
Since we have shown in the above arguments that
$\varepsilon_1,\dots,\varepsilon_r$ 
are determined by $\dim\varphi(c_j)E_{V^{[1]}}^\perp=\dim \varphi(c_j)E-\dim V_{j1}$ $(j=2,\dots,r)$ and by $\dim V_{kj}$ $(2\le j<k\le r)$,
the vector $\ue$ does not depend on the choice of $\eta$ and hence we have finished the proof.
\qed



\end{document}